\documentclass[a4paper,reqno,10pt]{amsart}

\usepackage{subfiles}
\usepackage{upgreek}
\usepackage{amsfonts,graphicx,bezier,enumerate,multicol,multirow,subfloat,color,xcolor,amssymb,mathtools,verbatim,appendix,pdfpages,color,colortbl,lscape,longtable,euscript}
\usepackage{amsmath,amsthm,amssymb,colonequals,amscd}
\usepackage{alphalph,etoolbox}
\usepackage{indentfirst}
\usepackage{amscd}
\usepackage{setspace}
\usepackage{mathrsfs}
\usepackage{float}
\usepackage{graphicx, subcaption}
\usepackage{wrapfig}
\usepackage{enumitem}
\usepackage{colonequals}
\usepackage{soul}
\usepackage{amsmath}
\usepackage{icomma}
\usepackage{cancel}
\usepackage{soul}
\setlist[enumerate]{format=\normalfont}

\newtheorem{theorem}{Theorem}[section]
\newtheorem{prop}[theorem]{Proposition}
\newtheorem{lemma}[theorem]{Lemma}
\newtheorem{definition}[theorem]{Definition}
\newtheorem{cor}[theorem]{Corollary}

\newtheorem{conj}[theorem]{Conjecture}

\theoremstyle{definition}

\newtheorem{remark}[theorem]{Remark}

\newtheorem{notation}[theorem]{Notation}

\numberwithin{equation}{section}

\newcommand{\bmc}{\begin{multicols}}
	\newcommand{\emc}{\end{multicols}}
\usepackage[colorlinks]{hyperref}
\usepackage{bookmark}

\usepackage{tikz}
\usetikzlibrary{arrows,decorations.pathmorphing,decorations.pathreplacing,positioning,shapes.geometric,shapes.misc,decorations.markings,decorations.fractals,calc,patterns}

\tikzset{>=stealth',
	cvertex/.style={circle,draw=black,inner sep=1pt,outer sep=3pt},
	vertex/.style={circle,fill=black,inner sep=1pt,outer sep=3pt},
	star/.style={circle,fill=yellow,inner sep=0.75pt,outer sep=0.75pt},
	tvertex/.style={inner sep=1pt,font=\scriptsize},
	gap/.style={inner sep=0.5pt,fill=white}}

\tikzstyle{mybox} = [draw=black, fill=blue!10, very thick,
rectangle, rounded corners, inner sep=10pt, inner ysep=20pt]
\tikzstyle{boxtitle} =[fill=blue!50, text=white,rectangle,rounded corners]

\numberwithin{equation}{section}

\newcommand{\looptop}[2]{\xy \SelectTips{cm}{10}
	\POS(0,0) \endxy}

\hypersetup{
	colorlinks   = true,          
	urlcolor     = blue,          
	linkcolor    = purple,          
	citecolor   = blue             
}

\setlist[enumerate]{format=\normalfont}

\newcommand{\scrM}{\EuScript{M}}

\newcommand{\scrR}{\EuScript{R}}
\newcommand{\scrS}{\EuScript{S}}

\DeclareMathOperator{\Rep}{\mathrm{Rep}}

\begin{document}
	
	\title{Deformations and Simultaneous Resolution of Determinantal Surfaces}
	\author{Brian Makonzi}
	\address{Brian Makonzi, Department of Mathematics, Muni University, P.O Box 725 Arua, Uganda \& Department of Mathematics, Makerere University, P.O Box 7062 Kampala, Uganda \& The Mathematics and Statistics Building, University of Glasgow, University Place, Glasgow, G12 8QQ, UK} 
	\email{mckonzi@gmail.com}
	\begin{abstract}
		This paper uses reconstruction algebras	to construct simultaneous resolution of determinantal surfaces. The main new difference to the classical case is that, in addition to the quiver of the reconstruction algebra, certain noncommutative relations, namely those of the canonical algebra of Ringel, are required. All the relations of the reconstruction algebra except the canonical relation are then deformed, and these deformed relations together with variation of the GIT quotient achieve the simultaneous resolution. 	
	\end{abstract}
	
	\maketitle
	\parindent 20pt
	\parskip 0pt
	
	\maketitle

	\section{introduction}	
	Rational surface singularities play a fundamental role in various areas of mathematics, such as algebraic geometry and singularity theory.  In this paper, we specifically focus on constructing simultaneous resolution using noncommutative techniques, in the case of determinantal singularities. These are both non-Gorenstein and non-toric. The noncommutative framework we employ enables us to explore a wider spectrum of rational surface singularities, facilitating a deeper understanding of their properties and paving the way for future investigations.

	\subsection{Motivation and Background}	
	Grothendieck and Brieskorn \cite{Resolution} \cite{defspace} made significant contributions in constructing the deformation space for Kleinian singularities $\mathbb{C}^2/H$ with $H\leq\mathrm{SL}(2,\mathbb{C}).$ They established a connection between these singularities and the Weyl group $W$ of the corresponding simple simply--connected complex Lie group. Brieskorn \cite{defspace} successfully constructed the semiuniversal deformation $D \to \mathfrak{h}_\mathbb{C}/W$ of Kleinian singularities, and after base change via the action of the Weyl group as in the diagram below, the resulting morphism $Art \to \mathfrak{h}_\mathbb{C}$ is a family of singular surfaces, which resolves simultaneously \cite{defspace}.
	
	\[
	\begin{tikzpicture}
	\node (A) at (0,0) {$Art$};
	\node (B) at (2,0) {$D$};
	\node (a) at (0,-1) {$\mathfrak{h}_\mathbb{C}$};
	\node (b) at (2,-1) {$\mathfrak{h}_\mathbb{C}/W$};
	\draw[->] (A)--(B);
	\draw[->] (A)--(a);
	\draw[->] (a)--(b);
	\draw[->] (B)--(b);
	\end{tikzpicture}
	\]

	Kronheimer \cite{kronheimer} and Cassens-Slodowy \cite[\S3]{OnKleinianSing} utilized the McKay quiver as a means to construct the semiuniversal deformation of Kleinian singularities, along with their simultaneous resolutions, belonging to types $A_n,~D_n,~E_6,~E_7$ and $E_8$. Building upon this, Crawley-Boevey and Holland \cite{crawley1998noncommutative} later provided a reinterpretation of this approach in the context of the deformed preprojective algebra.
	
	\medskip  
	However, for non-Gorenstein surface quotient singularities, namely
	those $\mathbb{C}^2/H$ for small finite groups $H\leq\mathrm{GL}(2,\mathbb{C})$ that are not inside $\mathrm{SL}(2,\mathbb{C})$, the situation poses a more formidable challenge. Artin \cite{ArtinComp} revealed that within the deformation space of non-Gorenstein singularities, there exists a distinct component (the Artin component) which is the smooth irreducible base space $\mathrm{H}^1_\mathbb{C}/W$ in the diagram below.
	Over this base, there is a family of singular surfaces $D  \to  \mathrm{H}^1_\mathbb{C}/W$ which after finite base change by some appropriate Weyl group $W$ pulls back to give a family of singular surfaces $ Art \to \mathrm{H}^1_\mathbb{C},$  which again admits simultaneous resolution.
	
	\[
	\begin{tikzpicture}
	\node (A) at (0,0) {$Art$};
	\node (B) at (2,0) {$D$};
	\node (a) at (0,-1) {$\mathrm{H}^1_\mathbb{C}$};
	\node (b) at (2,-1) {$\mathrm{H}^1_\mathbb{C}/W$};
	\draw[->] (A)--(B);
	\draw[->] (A)--(a);
	\draw[->] (a)--(b);
	\draw[->] (B)--(b);
	\end{tikzpicture}
	\]

	Riemenschneider \cite{deformRational} computed the Artin component for cyclic quotient singularities, then later in \cite[\S5]{RieCyclic} he used the McKay quiver and special representations as described by Wunram \cite{Wunram} to give an alternative description. But simultaneous resolution is also not obtained using the McKay quiver perspective. Our previous work \cite{Makonzi} solves this by using reconstruction algebras to construct simultaneous resolution for cyclic groups.\\
	
	Building upon this success, this paper extends these techniques to	determinantal surface singularities, thereby broadening the scope of known examples with simultaneous resolution.

	\subsection{Main Result}\label{main result}
	This paper considers a class of rational surface singularities called  determinantal surface singularities. Consider $R,$ the determinantal singularity given as the quotient of $\mathbb{C}[v, w_1, w_2, w_3]$ by the $2 \times 2$ minors of the matrix
	\[\left(
	\begin{array}{ccccc}
	{w}_2&{w}_3 &{v^{p_2}}\\
	{v^{p_1}}&{{w}_3+{v^{p_3}}}&{w}_1
	\end{array}
	\right).\]
	\vspace{0.1cm}
	The graph of exceptional curves of the minimal resolution $X \rightarrow \text{Spec}~ R$ is the star-shaped graph in (\ref{s Veron dual graph}).\\
	
	The quiver $Q$ of the reconstruction algebra corresponding to Spec$R$ is recalled in \S\ref{Preliminaries}. For dimension vector $\updelta=(1,\hdots,1),$ consider $\scrR\colonequals \mathbb{C}[\Rep(\mathbb{C}Q,\updelta)]/I$ and crucially $I$ is the ideal generated by Ringel's \cite{RingelCanAlgebras} canonical relation (see (\ref{ReconAlgebraQuiver})). This carries a natural action of $G \colonequals \textstyle \prod_{q \in Q_0} \mathbb{C}^{\ast}$ where $Q_0$ denotes the set of vertices of $Q.$ As shown in \S\ref{ReconAlg}, $\scrR^G$ is generated by cycles. These generate a $\mathbb{C}$--algebra $\mathbb{C}[\mathsf{w} , \mathsf{v}],$ and they further satisfy determinantal relations (recalled in \S \ref{QDetform}). 
	
	\medskip
	Simultaneous resolution is then achieved by introducing the deformed reconstruction algebra (see \S\ref{Deformed}), which generalises the work of 
	Crawley-Boevey--Holland \cite{crawley1998noncommutative} on deformed preprojective algebras. In \S\ref{SimRes}, we construct a map $\uppi\colon \text{Spec}~\scrR^G \rightarrow \Delta,$ where $\Delta$ is an affine space defined in Notation \ref{Notationrel}. The fibre above the origin is the determinantal singularity Spec $R$ above (see Remark \ref{fibreabovegamma} later). The following is our main result, where $\upvartheta_0$ is a \emph{particular} choice of stability condition explained in \S\ref{ModuliofReconAlg}.

	\begin{theorem}[{\ref{thm: main}}]\label{thm: main intro}The diagram
		\[
		\begin{tikzpicture}
		\node (A) at (0,0) {$\Rep(\mathbb{C}Q/I,~\updelta)~ /\!\!\!\!/_{\upvartheta_0} \mathrm{GL}$};
		\node (B) at (4,0) {$\text{Spec}~ \mathcal{\scrR}^G$};
		\node (b) at (4,-2) {$\Delta$};
		\draw[->] (A)-- node[above]  {} (B);
		\draw[densely dotted,->] (A)-- node[below]  {$\upphi$} (b);
		\draw[->] (B)-- node[right]  {$\uppi$} (b);
		a\end{tikzpicture}
		\]
		is a simultaneous resolution of singularities in the sense that the morphism $\upphi$ is smooth, and $\uppi$ is flat.
	\end{theorem}
	
	The smoothness of the fibres of $\upphi$ is achieved using moduli spaces of the deformed reconstruction algebra $\Lambda_{\boldsymbol{\upgamma}}.$ These are introduced in \S\ref{Deformed}, and may be of independent interest. Note that this is significantly more difficult than for cyclic quotients, with the main difficulty being showing that $\scrM_{\upvartheta_0} \neq \emptyset$ (see Remark \ref{gammainDelta}). 
	\medskip
	
	This paper is organised as follows. Section \ref{Preliminaries} recalls the reconstruction algebra of star-shaped graphs. Section \ref{SecSimRes} introduces the deformed reconstruction algebra, and uses this in Section \ref{SimResSection} to achieve simultaneous resolution. Section \ref{repvar}
	gives a conjectural presentation of $\mathcal{\scrR}^G,$ although this is not required to prove the existence of simultaneous resolution.\\
	
	\subsection*{Conventions} Throughout we work over the complex numbers $\mathbb{C}.$ If $a$ and $b$ are arrows in a quiver, $ab$ denotes $a$ followed by $b,$ $t(a)$ is the tail of $a$ and $h(a)$ is the head of $a.$
	
	\subsection*{Acknowledgements} The author would like to thank the GRAID program for sponsoring his PhD studies, and the
	ERC Consolidator Grant 101001227 (MMiMMa) for funding his visit to Glasgow. Furthermore,  he would like to extend special thanks to his supervisors Michael Wemyss and David Ssevviiri for their exceptional guidance and encouragement, the Department of Mathematics, Makerere University and the School of Mathematics \& Statistics, University of Glasgow for providing a welcoming and supportive environment. The hospitality and conducive atmosphere have greatly facilitated the author's studies and research.
	
	\section{preliminaries}\label{Preliminaries}
	This section recalls the reconstruction algebra of star-shaped graphs following \cite{IyamaWemyss}. Consider the following star shaped dual graph, where there are $n \geq 3$ arms. 
	\[
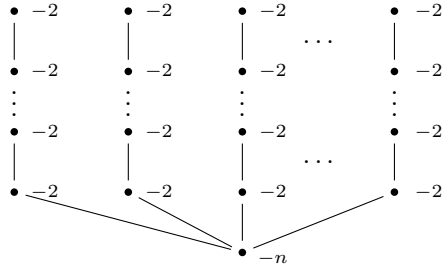
	\begin{array}{c}
	\begin{tikzpicture}[xscale=1,yscale=0.8]
	\node (0) at (0,0) [vertex] {};
	\node (A1) at (-3,1) [vertex]{};
	\node (A2) at (-3,2) [vertex] {};
	\node (A3) at (-3,3) [vertex] {};
	\node (A4) at (-3,4) [vertex] {};
	\node (B1) at (-1.5,1) [vertex] {};
	\node (B2) at (-1.5,2) [vertex] {};
	\node (B3) at (-1.5,3) [vertex] {};
	\node (B4) at (-1.5,4) [vertex] {};
	\node (C1) at (0,1) [vertex] {};
	\node (C2) at (0,2) [vertex] {};
	\node (C3) at (0,3) [vertex] {};
	\node (C4) at (0,4) [vertex] {};
	\node (n1) at (2,1) [vertex] {};
	\node (n2) at (2,2) [vertex] {};
	\node (n3) at (2,3) [vertex] {};
	\node (n4) at (2,4) [vertex] {};
	\node at (-3,2.6) {$\vdots$};
	\node at (-1.5,2.6) {$\vdots$};
	\node at (0,2.6) {$\vdots$};
	\node at (2,2.6) {$\vdots$};
	\node at (1,3.5) {$\hdots$};
	\node at (1,1.5) {$\hdots$};
	\node (T) at (0,4.25) {};
	\node at (-2.6,1) {$\scriptstyle -2$};
	\node at (-2.6,2) {$\scriptstyle -2$};
	\node at (-2.6,3) {$\scriptstyle -2$};
	\node at (-2.6,4) {$\scriptstyle -2$};
	\node at (-1.1,1) {$\scriptstyle -2$};
	\node at (-1.1,2) {$\scriptstyle -2$};
	\node at (-1.1,3) {$\scriptstyle -2$};
	\node at (-1.1,4) {$\scriptstyle -2$};
	\node at (0.4,1) {$\scriptstyle -2$};
	\node at (0.4,2) {$\scriptstyle -2$};
	\node at (0.4,3) {$\scriptstyle -2$};
	\node at (0.4,4) {$\scriptstyle -2$};
	\node at (2.45,1) {$\scriptstyle -2$};
	\node at (2.45,2) {$\scriptstyle -2$};
	\node at (2.45,3) {$\scriptstyle -2$};
	\node at (2.45,4) {$\scriptstyle -2$};
	\node at (0.4,-0.1) {$\scriptstyle -n$};
	\draw (A1) -- (0);
	\draw (B1) -- (0);
	\draw (C1) -- (0);
	\draw (n1) -- (0);
	\draw (A2) -- (A1);
	\draw (B2) -- (B1);
	\draw (C2) -- (C1);
	\draw (n2) -- (n1);
	\draw (A4) -- (A3);
	\draw (B4) -- (B3);
	\draw (C4) -- (C3);
	\draw (n4) -- (n3);
	\end{tikzpicture}
	\end{array}
	\]
	
	\noindent
	Below, we restrict to the case $n = 3.$ Note that all techniques below generalise, at the cost of additional significant notation. Thus our dual graph is 
	
	\begin{equation}\label{s Veron dual graph}
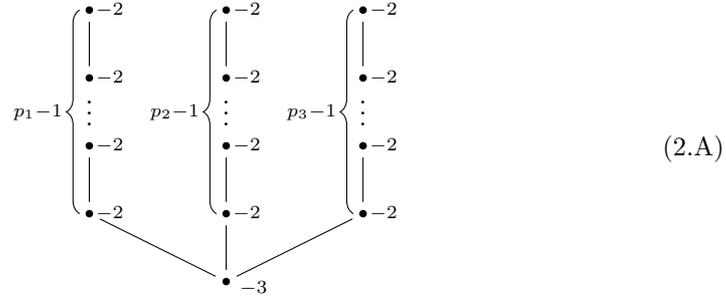
	\begin{array}{c}
	\begin{tikzpicture}[xscale=0.9,yscale=0.9]
	\node (0) at (0,0) [vertex] {};
	\node (B1) at (-2,1) [vertex] {};
	\node (B2) at (-2,2) [vertex] {};
	\node (B3) at (-2,3) [vertex] {};
	\node (B4) at (-2,4) [vertex] {};
	\node (C1) at (0,1) [vertex] {};
	\node (C2) at (0,2) [vertex] {};
	\node (C3) at (0,3) [vertex] {};
	\node (C4) at (0,4) [vertex] {};
	\node (n1) at (2,1) [vertex] {};
	\node (n2) at (2,2) [vertex] {};
	\node (n3) at (2,3) [vertex] {};
	\node (n4) at (2,4) [vertex] {};
	\node at (-2,2.6) {$\vdots$};
	\node at (0,2.6) {$\vdots$};
	\node at (2,2.6) {$\vdots$};
	\node (T) at (0,4.25) {};
	\node at (-1.7,1) {$\scriptstyle -2$};
	\node at (-1.7,2) {$\scriptstyle -2$};
	\node at (-1.7,3) {$\scriptstyle -2$};
	\node at (-1.7,4) {$\scriptstyle -2$};
	\node at (0.3,1) {$\scriptstyle -2$};
	\node at (0.3,2) {$\scriptstyle -2$};
	\node at (0.3,3) {$\scriptstyle -2$};
	\node at (0.3,4) {$\scriptstyle -2$};
	\node at (2.3,1) {$\scriptstyle -2$};
	\node at (2.3,2) {$\scriptstyle -2$};
	\node at (2.3,3) {$\scriptstyle -2$};
	\node at (2.3,4) {$\scriptstyle -2$};
	\node at (0.4,-0.1) {$\scriptstyle -3$};
	\draw (B1) -- (0);
	\draw (C1) -- (0);
	\draw (n1) -- (0);
	\draw (B2) -- (B1);
	\draw (C2) -- (C1);
	\draw (n2) -- (n1);
	\draw (B4) -- (B3);
	\draw (C4) -- (C3);
	\draw (n4) -- (n3);
	\draw [decorate,decoration={brace,amplitude=5pt},xshift=-4pt,yshift=0pt]
	(2,1) -- (2,4) node [black,midway,xshift=-0.55cm] 
	{$\scriptstyle p_3-1$};
	\draw [decorate,decoration={brace,amplitude=5pt},xshift=-4pt,yshift=0pt]
	(0,1) -- (0,4) node [black,midway,xshift=-0.55cm] 
	{$\scriptstyle p_2-1$};
	\draw [decorate,decoration={brace,amplitude=5pt},xshift=-4pt,yshift=0pt]
	(-2,1) -- (-2,4) node [black,midway,xshift=-0.55cm] 
	{$\scriptstyle p_1-1$};
	\end{tikzpicture}
	\end{array}
	\end{equation}
	where there are $3$ arms, and the number of vertices on arm $i$ is $p_i-1$. Contracting these curves gives the singularity Spec$R$ in \S \ref{main result}. To (\ref{s Veron dual graph}) we associate the extended quiver, and the extended vertex is called the $0^{th}$ vertex. The resulting quiver is as shown below 
	
	\[
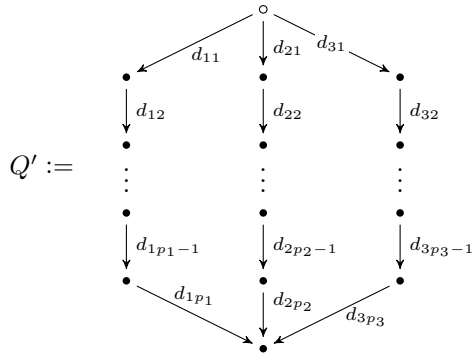
	\begin{array}{cc}
	Q^{\prime} \colonequals  &
	\begin{array}{c}
	\\
	\begin{tikzpicture}[xscale=0.9,yscale=0.9,bend angle=40, looseness=1]
	\node (0) at (0,0) [vertex] {};
	\node (B1) at (-2,1) [vertex] {};
	\node (B2) at (-2,2) [vertex] {};
	\node (B3) at (-2,3) [vertex] {};
	\node (B4) at (-2,4) [vertex] {};
	\node (C1) at (0,1) [vertex] {};
	\node (C2) at (0,2) [vertex] {};
	\node (C3) at (0,3) [vertex] {};
	\node (C4) at (0,4) [vertex] {};
	\node (n1) at (2,1) [vertex] {};
	\node (n2) at (2,2) [vertex] {};
	\node (n3) at (2,3) [vertex] {};
	\node (n4) at (2,4) [vertex] {};
	\node at (-2,2.6) {$\vdots$};
	\node at (0,2.6) {$\vdots$};
	\node at (2,2.6) {$\vdots$};
	\node (T) at (0,5) [cvertex] {};
	\draw [->] (B1) --node[above,pos=0.5]{$\scriptstyle d_{1p_1}$} (0);
	\draw [->] (C1) --node[right,pos=0.2]{$\scriptstyle d_{2p_2}$} (0);
	\draw [->] (n1) --node[below,pos=0.2]{$\scriptstyle d_{3p_3}$} (0);
	\draw [->] (B2) --node[right,pos=0.5]{$\scriptstyle d_{1p_1-1}$}(B1);
	\draw [->] (C2) --node[right,pos=0.5]{$\scriptstyle d_{2p_2-1}$}(C1);
	\draw [->] (n2) --node[right,pos=0.5]{$\scriptstyle d_{3p_3-1}$} (n1);
	\draw [->] (B4) --node[right,pos=0.5]{$\scriptstyle d_{12}$}(B3);
	\draw [->] (C4) --node[right,pos=0.5]{$\scriptstyle d_{22}$}(C3);
	\draw [->] (n4) -- node[right,pos=0.5]{$\scriptstyle d_{32}$} (n3);
	\draw [->] (T) -- node[below,pos=0.4]{$\scriptstyle d_{11}$}(B4);
	\draw [->] (T) -- node[right,pos=0.6]{$\scriptstyle d_{21}$}(C4);
	\draw [->] (T) -- node[gap,pos=0.5]{$\scriptstyle d_{31}$}(n4);
	\end{tikzpicture}
	\end{array}
	\end{array}
	\]
	The notation $d_{12}$ is read `the second downward arrow along arm 1'.\\
	
	From this, Ringel \cite{RingelCanAlgebras} introduces the canonical algebra, as the path algebra of the quiver $Q^{\prime}$ subject to the relation $$I = \left<d_{11}\hdots d_{1p_1} - d_{21}\hdots d_{2p_2} + d_{31}\hdots d_{3p_3}\right>.$$
	\noindent

	The double quiver of $Q^{\prime}$ is written as follows
	
	\begin{equation}\label{ReconAlgebraQuiver}
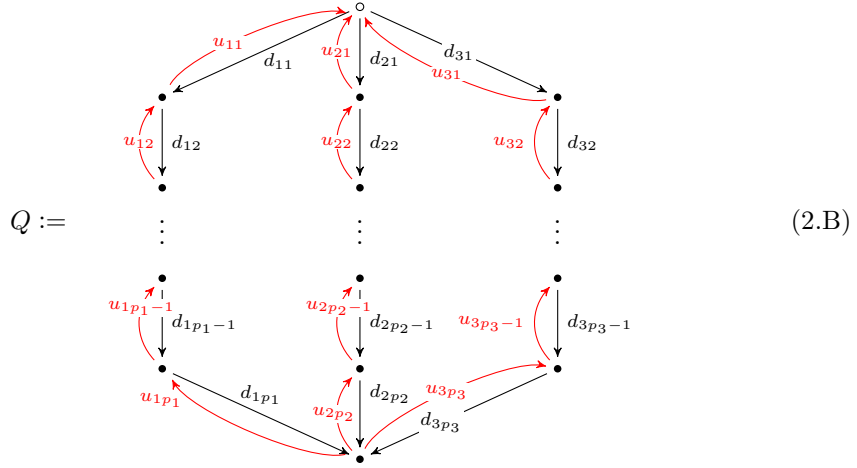
	\begin{array}{cc}
	Q \colonequals  &
	\begin{array}{c}
	\\
	\begin{tikzpicture}[xscale=1.3,yscale=1.2,bend angle=40, looseness=1]
	\node (0) at (0,0) [vertex] {};
	\node (B1) at (-2,1) [vertex] {};
	\node (B2) at (-2,2) [vertex] {};
	\node (B3) at (-2,3) [vertex] {};
	\node (B4) at (-2,4) [vertex] {};
	\node (C1) at (0,1) [vertex] {};
	\node (C2) at (0,2) [vertex] {};
	\node (C3) at (0,3) [vertex] {};
	\node (C4) at (0,4) [vertex] {};
	\node (n1) at (2,1) [vertex] {};
	\node (n2) at (2,2) [vertex] {};
	\node (n3) at (2,3) [vertex] {};
	\node (n4) at (2,4) [vertex] {};
	\node at (-2,2.6) {$\vdots$};
	\node at (0,2.6) {$\vdots$};
	\node at (2,2.6) {$\vdots$};
	\node (T) at (0,5) [cvertex] {};
	\draw [->] (B1) --node[above,pos=0.5]{$\scriptstyle d_{1p_1}$} (0);
	\draw [->] (C1) --node[right,pos=0.2]{$\scriptstyle d_{2p_2}$} (0);
	\draw [->] (n1) --node[gap,pos=0.6]{$\scriptstyle d_{3p_3}$} (0);
	\draw [->] (B2) --node[right,pos=0.5]{$\scriptstyle d_{1p_1-1}$}(B1);
	\draw [->] (C2) --node[right,pos=0.5]{$\scriptstyle d_{2p_2-1}$}(C1);
	\draw [->] (n2) --node[right,pos=0.5]{$\scriptstyle d_{3p_3-1}$} (n1);
	\draw [->] (B4) --node[right,pos=0.5]{$\scriptstyle d_{12}$}(B3);
	\draw [->] (C4) --node[right,pos=0.5]{$\scriptstyle d_{22}$}(C3);
	\draw [->] (n4) -- node[right,pos=0.5]{$\scriptstyle d_{32}$} (n3);
	\draw [->] (T) -- node[below,pos=0.4]{$\scriptstyle d_{11}$}(B4);
	\draw [->] (T) -- node[right,pos=0.6]{$\scriptstyle d_{21}$}(C4);
	\draw [->] (T) -- node[gap,pos=0.5]{$\scriptstyle d_{31}$}(n4);
	\draw [bend right, bend angle=10, looseness=0.5, <-, red] (B1)+(-50:4.5pt) to node[left,pos=0.2]{$\scriptstyle u_{1p_1}$}($(0) + (160:4.5pt)$);
	\draw [bend right, <-, red](C1) to node[gap,pos=0.5]{$\hspace{-0.2em}\scriptstyle u_{2p_2}$} (0);
	\draw [bend right, bend angle=10, looseness=0.5, <-, red](n1) to node[gap,pos=0.5]{$\scriptstyle u_{3p_3}$} (0);
	\draw [bend right, <-, red] (B2) to node[gap,pos=0.3]{$\scriptstyle u_{1p_1-1}$}(B1);
	\draw [bend right, <-, red] (C2) to node[gap,pos=0.3]{$\scriptstyle u_{2p_2-1}$}(C1);
	\draw [bend right, <-, red] (n2) to node[left,pos=0.5]{$\scriptstyle u_{3p_3-1}$}(n1);
	\draw [bend right, <-, red] (B4) to node[gap,pos=0.5]{$\scriptstyle u_{12}$}(B3);
	\draw [bend right, <-, red] (C4) to node[gap,pos=0.5]{$\scriptstyle u_{22}$}(C3);
	\draw [bend right, <-, red] (n4) to node[left,pos=0.5]{$\scriptstyle u_{32}$}(n3);
	\draw [bend right, bend angle=10, looseness=0.5, <-, red] (T)+(-155:4.5pt) to node[gap,pos=0.6]{$\scriptstyle u_{11}$}($(B4) + (60:4.5pt)$);
	\draw [bend right, <-, red] (T) to node[gap,pos=0.5]{$\scriptstyle u_{21}$}(C4);
	\draw [bend right, bend angle=10, looseness=0.5,  <-, red] (T) to node[gap,pos=0.5]{$\scriptstyle u_{31}$}(n4);
	\end{tikzpicture}
	\end{array}
	\end{array}
	\end{equation}
	Set $D_i=d_{i1}\hdots d_{ip_i},$ $U_i=u_{ip_i}\hdots u_{i1}$, which are the paths down, respectively up, arm $i$.

	\begin{definition}\label{reconalgebradefn}
	The reconstruction algebra is defined to be the path algebra of $Q,$ subject to the relations induced by the canonical relations, and at every vertex, all 2-cycles that start and stop at that vertex are equal. Thus, the recontruction algebra can be written explicitly as the following relations.
	\vspace{-0.1cm} 
	\begin{align*}
	u_{1i}d_{1i} &= d_{1i+1}u_{1i+1}  ~ \text{for all}~ 1 \leq i \leq p_1-1\\
	u_{2i}d_{2i} &= d_{2i+1}u_{2i+1}   ~\text{for all}~ 1 \leq i \leq p_2-1\\
	u_{3i}d_{3i} &= d_{3i+1}u_{3i+1}  ~ \text{for all}~ 1 \leq i \leq p_3-1\\
	d_{21}u_{21} &= d_{11}u_{11} \\
	d_{21}u_{21} &= d_{31}u_{31} \\
	u_{1p_1}d_{1p_1} &= u_{2p_2}d_{2p_2}\\
	u_{3p_3}d_{3p_3} &= u_{2p_2}d_{2p_2}\\
	D_1 - D_2 &+ D_3 = 0.
	\end{align*}	
	\end{definition}

	\section{The Moduli Space}\label{SecSimRes}
	In this section, the moduli space of representations of quivers of both the reconstruction algebra and the deformed reconstruction algebra is computed. This will be used to achieve simultaneous resolution in \S\ref{SimResSection}.
	
	\subsection{Generalities}\label{Generalities}	
	Consider the dimension vector $\updelta=(1,\hdots,1)$ and the representation variety $\Rep(\mathbb{C}Q,\updelta)$, where $Q$ is the quiver in \eqref{ReconAlgebraQuiver}. We will abuse notation, and for any arm $i$ also write $D_i, d_{i1}, \hdots, d_{ip_i}, U_i, u_{ip_i}, \hdots, u_{i1}$ for the linear maps associated to the arrows with those labels. We consider $\Rep(\mathbb{C}Q,\updelta)$ as an affine space, and write 
	
	\begin{equation}\label{eqnRep}
	\scrR\colonequals \frac{\mathbb{C}[\Rep(\mathbb{C}Q,\updelta)]}{\left<D_1 -  D_2 + D_3  \right>,}
	\end{equation}
	\noindent
	which we identify with the polynomial ring in the number of arrow variables subject to the relation
	
	$$I = \left<D_1 - D_2 + D_3  \right>.$$
	
	The co-ordinate ring carries a natural action of $ G \colonequals \textstyle \prod_{q \in Q_0} \mathbb{C}^{\ast}$ where $Q_0$ denotes the set of vertices of $Q$. The action is via conjugation, namely $ \upmu \in G = \mathbb{C}^\ast \times  \hdots \times \mathbb{C}^\ast$ acts on  $p + I \in \scrR$ as $\upmu \cdot (p + I) \colonequals \upmu \cdot p + I = \upmu^{-1}_{t(p)} p \upmu_{h(p)} + I$. 
	
	\subsection{Moduli of Quiver with one relation}\label{ModuliofReconAlg}
	With respect to the ordering of the vertices as in Section \ref{Preliminaries}, this subsection computes the moduli space of representations of the quiver $Q$ in \S\ref{Preliminaries} subject to the single relation $I = \left<D_1 -  D_2 + D_3 \right>,$ under the dimension vector $\updelta = (1,1, \hdots, 1),$ and King stability condition $\upvartheta_0 = (-(m_1+m_2+m_3 +1),1, \hdots, 1).$\\

	Set $\scrM_{\upvartheta_0} := \scrR~ 
	/\!\!\!\!/_{\upvartheta_0} \mathrm{GL}$ where $\scrR$ is defined in \eqref{eqnRep}, and consider the following open subsets of $\scrM_{\upvartheta_0}$ with either $D_1, ~ D_2~ \text{or} ~D_3 \neq 0$ where $D_i \coloneqq d_{i1} \hdots d_{ip_i},$
	
	\[V^1_{i,j} = \left\{
	\begin{array}{cccc}
	D_1 \neq 0, & d_{21} \hdots d_{2i} \neq 0, & u_{2p_2} \hdots u_{2{i+2}} \neq 0 \\
	&d_{31} \hdots d_{3j} \neq 0, & u_{3p_3} \hdots u_{3{j+2}} \neq 0 
	\end{array}
	\right\}  0 \leq i \leq p_2-1,~
	0 \leq j \leq p_3-1, \]
	
	\[V^2_{i,j} = \left\{
	\begin{array}{cccc}
	D_2 \neq 0, & d_{11} \hdots d_{1i} \neq 0, & u_{1p_1} \hdots u_{1{i+2}} \neq 0 \\
	&d_{31} \hdots d_{3j} \neq 0, & u_{3p_3} \hdots u_{3{j+2}} \neq 0 \\
	\end{array}
	\right\} 0 \leq i \leq p_1-1,~0 \leq j \leq p_3-1,\]
	
	\[V^3_{i,j} = \left\{
	\begin{array}{cccc}
	D_3 \neq 0, & d_{11} \hdots d_{1i} \neq 0, & u_{1p_1} \hdots u_{1{i+2}} \neq 0 \\
	&d_{21} \hdots d_{2j} \neq 0, & u_{2p_2} \hdots u_{2{j+2}} \neq 0 \\
	\end{array}
	\right\} 0 \leq i \leq p_1-1,~0 \leq j \leq p_2-1,\]

\noindent  	
	where by convention, if a product is zero (e.g $d_{21}\hdots d_{2i}$ with $i=0,$ or $u_{3p_3}\hdots u_{3j+2}$ with $j=p_3-1$), then that condition is empty.
	
	\begin{lemma}\label{Cover}
		With notation as above, the open sets $V^1_{i,j}, V^2_{i,j}, V^3_{i,j}$ completely cover the moduli space $\scrM_{\upvartheta_0}$.
	\end{lemma}
	
	\begin{proof}
		By \cite[Remark 3.10]{NCCR}, a $\mathbb{C}Q/I$ module $M$ of dimension vector $(1,1, \hdots, 1)$ is $\upvartheta_0$-stable if and only if
		for every vertex in the quiver representation of $M$, there exists a non-zero path from the zero vertex to that vertex. So, if we consider the bottom vertex of $M,$ we must have either $D_1\neq 0, ~ D_2\neq 0$ or $D_3 \neq 0.$ By symmetry, suppose that $D_1\neq 0.$ For convenience in the argument below, label the arrows and vertices of $Q$ as follows
		
		\[
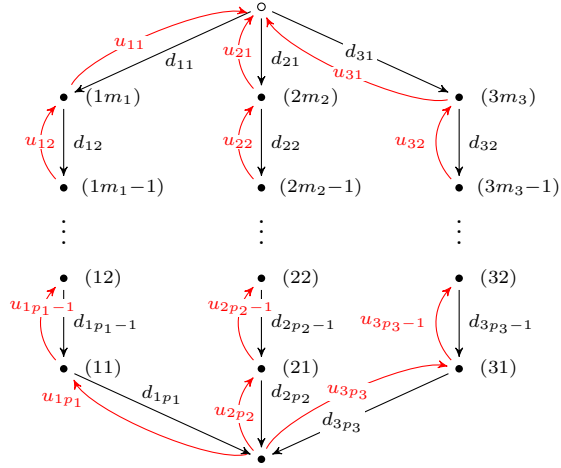
		\begin{array}{cc}
		 &
		\begin{array}{c}
		\\
		\begin{tikzpicture}[xscale=1.3,yscale=1.2,bend angle=40, looseness=1]
		\node (0) at (0,0) [vertex] {};
		\node (B1) at (-2,1) [vertex] {};
		\node (B2) at (-2,2) [vertex] {};
		\node (B3) at (-2,3) [vertex] {};
		\node (B4) at (-2,4) [vertex] {};
		\node (C1) at (0,1) [vertex] {};
		\node (C2) at (0,2) [vertex] {};
		\node (C3) at (0,3) [vertex] {};
		\node (C4) at (0,4) [vertex] {};
		\node (n1) at (2,1) [vertex] {};
		\node (n2) at (2,2) [vertex] {};
		\node (n3) at (2,3) [vertex] {};
		\node (n4) at (2,4) [vertex] {};
			\node at (-1.6,1) {$\scriptstyle (11)$};
		\node at (-1.6,2) {$\scriptstyle (12)$};
		\node at (-1.4,3) {$\scriptstyle  (1m_1-1)$};
		\node at (-1.5,4) {$\scriptstyle  (1m_1)$};
		\node at (0.4,1) {$\scriptstyle  (21)$};
		\node at (0.4,2) {$\scriptstyle  (22)$};
		\node at (0.6,3) {$\scriptstyle  (2m_2-1)$};
		\node at (0.5,4) {$\scriptstyle  (2m_2)$};
		\node at (2.4,1) {$\scriptstyle (31)$};
		\node at (2.4,2) {$\scriptstyle (32)$};
		\node at (2.6,3) {$\scriptstyle (3m_3-1)$};
		\node at (2.5,4) {$\scriptstyle (3m_3)$};
		\node at (-2,2.6) {$\vdots$};
		\node at (0,2.6) {$\vdots$};
		\node at (2,2.6) {$\vdots$};
		\node (T) at (0,5) [cvertex] {};
		\draw [->] (B1) --node[above,pos=0.5]{$\scriptstyle d_{1p_1}$} (0);
		\draw [->] (C1) --node[right,pos=0.2]{$\scriptstyle d_{2p_2}$} (0);
		\draw [->] (n1) --node[gap,pos=0.6]{$\scriptstyle d_{3p_3}$} (0);
		\draw [->] (B2) --node[right,pos=0.5]{$\scriptstyle d_{1p_1-1}$}(B1);
		\draw [->] (C2) --node[right,pos=0.5]{$\scriptstyle d_{2p_2-1}$}(C1);
		\draw [->] (n2) --node[right,pos=0.5]{$\scriptstyle d_{3p_3-1}$} (n1);
		\draw [->] (B4) --node[right,pos=0.5]{$\scriptstyle d_{12}$}(B3);
		\draw [->] (C4) --node[right,pos=0.5]{$\scriptstyle d_{22}$}(C3);
		\draw [->] (n4) -- node[right,pos=0.5]{$\scriptstyle d_{32}$} (n3);
		\draw [->] (T) -- node[below,pos=0.4]{$\scriptstyle d_{11}$}(B4);
		\draw [->] (T) -- node[right,pos=0.6]{$\scriptstyle d_{21}$}(C4);
		\draw [->] (T) -- node[gap,pos=0.5]{$\scriptstyle d_{31}$}(n4);
		\draw [bend right, bend angle=10, looseness=0.5, <-, red] (B1)+(-50:4.5pt) to node[left,pos=0.2]{$\scriptstyle u_{1p_1}$}($(0) + (160:4.5pt)$);
		\draw [bend right, <-, red](C1) to node[gap,pos=0.5]{$\hspace{-0.2em}\scriptstyle u_{2p_2}$} (0);
		\draw [bend right, bend angle=10, looseness=0.5, <-, red](n1) to node[gap,pos=0.5]{$\scriptstyle u_{3p_3}$} (0);
		\draw [bend right, <-, red] (B2) to node[gap,pos=0.3]{$\scriptstyle u_{1p_1-1}$}(B1);
		\draw [bend right, <-, red] (C2) to node[gap,pos=0.3]{$\scriptstyle u_{2p_2-1}$}(C1);
		\draw [bend right, <-, red] (n2) to node[left,pos=0.5]{$\scriptstyle u_{3p_3-1}$}(n1);
		\draw [bend right, <-, red] (B4) to node[gap,pos=0.5]{$\scriptstyle u_{12}$}(B3);
		\draw [bend right, <-, red] (C4) to node[gap,pos=0.5]{$\scriptstyle u_{22}$}(C3);
		\draw [bend right, <-, red] (n4) to node[left,pos=0.5]{$\scriptstyle u_{32}$}(n3);
		\draw [bend right, bend angle=10, looseness=0.5, <-, red] (T)+(-155:4.5pt) to node[gap,pos=0.6]{$\scriptstyle u_{11}$}($(B4) + (60:4.5pt)$);
		\draw [bend right, <-, red] (T) to node[gap,pos=0.5]{$\scriptstyle u_{21}$}(C4);
		\draw [bend right, bend angle=10, looseness=0.5,  <-, red] (T) to node[gap,pos=0.5]{$\scriptstyle u_{31}$}(n4);
		 [decorate,decoration={brace,amplitude=5pt,mirror},xshift=4pt,yshift=0pt]
		\end{tikzpicture}
		\end{array}
		\end{array}
		\]
		
		\smallskip
		\noindent
		If $d_{21} = 0$ and $d_{31} = 0$ then the only way a non-zero path can reach vertex $(2m_2)$ and vertex $(3m_3)$ is if $D_1u_{2p_2} \hdots u_{22} \neq 0$ and $D_1u_{3p_3} \hdots u_{32} \neq 0$ respectively. In this case $M$ belongs to $V^1_{0,0}.$ If $d_{21} \hdots  d_{2p_2-1} \neq 0$ and $d_{31} \hdots  d_{3p_3-1} \neq 0$ then by a similar argument $M$ is in $V^1_{p_2-1,p_3-1}.$ Hence we can assume that some $d$ on arm 2, and some $d$ on arm 3 is equal to zero. In each arm, consider the $d$ closest to the top vertex which is zero: thus
		\begin{align*}
		d_{21} \hdots  d_{2i} \neq 0 &~\text{and}~ d_{2i+1} = 0\\
		d_{31} \hdots  d_{3j} \neq 0 &~\text{and} ~ d_{3j+1} = 0,
		\end{align*}
		and therefore $M$ is in $V^1_{i,j}.$

	\end{proof}
	\noindent
	From now on, to ease notation set $\mathbb{C}[U^k_{i,j}] = \mathbb{C}[V^k_{i-1,j-1}]$ for $k = 1, 2, 3.$

	\begin{prop}\label{hypersurfaces} With notation as above, the following statements hold.
		\begin{align*}
		\mathbb{C}[U^1_{i,j}] \cong \frac{\mathbb{C}[u_{11}, \hdots, u_{1p_1}, u_{21}, \hdots,  u_{2i}, d_{2i}, \hdots , d_{2p_2}, u_{31}, \hdots,  u_{3j}, d_{3j}, \hdots , d_{3p_3}]}{\left( 1 -  d_{2i} \hdots  d_{2p_2} + d_{3j} \hdots  d_{3p_3}\right)},\\[2mm]
		\mathbb{C}[U^2_{i,j}] \cong \frac{\mathbb{C}[u_{21}, \hdots, u_{2p_2}, u_{11}, \hdots,  u_{1i}, d_{1i}, \hdots , d_{1p_1}, u_{31}, \hdots,  u_{3j}, d_{3j}, \hdots , d_{3p_3}]}{\left( d_{1i} \hdots  d_{1p_1} -  1 + d_{3j} \hdots  d_{3p_3}\right)},\\[2mm]
		\mathbb{C}[U^3_{i,j}] \cong \frac{\mathbb{C}[u_{31}, \hdots, u_{3p_3}, u_{11}, \hdots,  u_{1i}, d_{1i}, \hdots , d_{1p_1}, u_{21}, \hdots,  u_{2j}, d_{2j}, \hdots , d_{2p_2}]}{\left( d_{1i} \hdots  d_{1p_1} -  d_{2j} \hdots  d_{2p_2} + 1\right)}.
		\end{align*}
	\end{prop}
	\normalsize
	\begin{proof}\emph{Case 1.} For $U^1_{i,j}$, after changing basis we may (and do) set the specified non-zero arrows to be $1$.  Exactly as in e.g. \cite[\S4]{TypeA}, this then exhausts the action of the group $\mathrm{GL}$, and thus we may read off the co-ordinates of the chart directly.
		\noindent	
		The only possible non-identity arrows on arm 1 are 
		\[u_{11}, \hdots, u_{1p_1},\] 
		
		\noindent
		the only possible non-identity arrows on arm 2 are 
		\[u_{21}, \hdots, u_{2i-1}, u_{2i}, d_{2i}, d_{2i+1}, \hdots, d_{2p_2},\]
		
		\noindent
		and the only possible non-identity arrows on arm 3 are 
		\[u_{31}, \hdots, u_{3j-1}, u_{3j}, d_{3j}, d_{3j+1}, \hdots , d_{3p_3}.\]
		
		\noindent
		Since
		$D_1=1,~D_2=d_{2i}d_{2i+1} \hdots d_{2p_2},~D_3=d_{3j}d_{3j+1} \hdots  d_{3p_3}$ the only relation is $1 -  d_{2i} \hdots  d_{2p_2} + d_{3j} \hdots  d_{3p_3},$ as claimed.

		\smallskip
		\noindent
		\emph{Case 2.} The proof of $\mathbb{C}[U^2_{i,j}]$ and $\mathbb{C}[U^3_{i,j}]$ are similar to that of $\mathbb{C}[U^1_{i,j}]$ above.
	\end{proof}
	
	\begin{cor}\label{ModuliSmooth}
		$\scrM_{\upvartheta_0}$ is smooth.
	\end{cor}
	
	\begin{proof}
		By Lemma \ref{Cover}, $\scrM_{\upvartheta_0}$ is covered by the open sets $U^1_{i,j}, U^2_{i,j}, U^3_{i,j},$ so it suffices to prove that each of these is smooth. By Proposition \ref{hypersurfaces}, set $f =1 -  d_{2i} \hdots  d_{2p_2} + d_{3j} \hdots  d_{3p_3}$ and consider the ideal
		\[K= \left(f,~\frac{\partial f}{\partial v}\mid v~ \text{is variable in}~ \mathbb{C}[U^1_{i,j}]\right).\]
		\noindent
		After taking partial derivatives with respect to $d_{2i}$ and $d_{3j}$, for example, it becomes clear that $1\in K$, so as is standard \cite[Algorithm 5.7.6]{Greuel}, this implies that $\mathbb{C}[U^1_{i,j}]$ is smooth.  The proof for $\mathbb{C}[U^2_{i,j}]$ and $\mathbb{C}[U^3_{i,j}]$ is identical.
		Thus, $\scrM_{\upvartheta_0}$ is smooth.
	\end{proof}

	\subsection{The Deformed Reconstruction Algebras}\label{Deformed}
	In this subsection we introduce a deformed version of the reconstruction algebra, and compute its moduli space $\scrM_{\upvartheta_0}(\Lambda_{\boldsymbol{\upgamma}},~\updelta)$.
	\begin{definition}\label{DeformedReconAlgbera}
		Given scalars $\boldsymbol{\upgamma} \in \mathbb{C}^{\oplus p_1-1} \oplus \mathbb{C}^{\oplus  p_2 -1 } \oplus \mathbb{C}^{\oplus  p_3 -1 } \oplus \mathbb{C}^{\oplus  4 },$ write $\boldsymbol{\upgamma} = (\boldsymbol{\upgamma}_{1}, \boldsymbol{\upgamma}_{2}, \boldsymbol{\upgamma}_{3},$ $A, B, a, b) ~\text{where}~ \boldsymbol{\upgamma}_{1} = (\upgamma_{11}, \hdots,\upgamma_{1p_1-1} ), \boldsymbol{\upgamma}_{2} = (\upgamma_{21}, \hdots,\upgamma_{2p_2-1} ), \boldsymbol{\upgamma}_{3} = (\upgamma_{31}, \hdots,\upgamma_{3p_3-1} ).$ Then the deformed reconstruction algebra $\Lambda_{\boldsymbol{\upgamma}}$ is defined to be the path algebra of the quiver $Q,$ subject to the following relations. 	
	\end{definition}
	\vspace{-0.7cm}\[
	\begin{array}{ccccc}
	(1)&~u_{1i}d_{1i} - d_{1i+1}u_{1i+1} &=  &\upgamma_{1i} &~ \text{for all}~ 1 \leq i \leq p_1-1\\
	(2)&~u_{2i}d_{2i} - d_{2i+1}u_{2i+1} &=  &\upgamma_{2i} &~ \text{for all}~ 1 \leq i \leq p_2-1\\
	(3)&~u_{3i}d_{3i} - d_{3i+1}u_{3i+1} &=  &\upgamma_{3i}&~ \text{for all}~ 1 \leq i \leq p_3-1\\
	(a)&~d_{21}u_{21} - d_{11}u_{11} &=  &a&\\
	(b)&~d_{21}u_{21} - d_{31}u_{31} &=  &b&\\
	(c)&~u_{1p_1}d_{1p_1} - u_{2p_2}d_{2p_2} &=  &A&\\
	(d)&~u_{3p_3}d_{3p_3} - u_{2p_2}d_{2p_2} &= &B&\\
	(x)&~d_{11} \hdots d_{1p_1} - d_{21} \hdots d_{2p_2} + d_{31} \hdots d_{3p_3}&= &0.&
	\end{array}\]  
	Note that $\Lambda_{\boldsymbol{0}}$ is the reconstruction algebra defined in Definition \ref{reconalgebradefn} earlier.
	\begin{notation}\label{Notationrel}
		Set
		\[
		\Delta := \{\boldsymbol{\upgamma} \in \mathbb{C}^{\oplus p_1 + p_2 + p_3 + 1}  \mid \sum_{i=1}^{p_{1}-1} \upgamma_{1i} - \sum_{i=1}^{p_{2}-1} \upgamma_{2i} + A + a =0 ~\text{and}~ \sum_{i=1}^{p_{3}-1} \upgamma_{3i} - \sum_{i=1}^{p_{2}-1} \upgamma_{2i} + B + b =0\}\label{def:Delta}.
		\]	
	\end{notation}
	
	\begin{remark}
		If $\boldsymbol{\upgamma} \in \Delta,$ then $\sum_{i=1}^{p_{1}-1} \upgamma_{1i} - \sum_{i=1}^{p_{3}-1} \upgamma_{3i} + (A-B) + (a-b) =0.$
	\end{remark}
	
	\begin{proof}
		For $\boldsymbol{\upgamma} \in \Delta,$
		\[ \sum_{i=1}^{p_{1}-1} \upgamma_{1i} - \sum_{i=1}^{p_{2}-1} \upgamma_{2i} + A + a= \sum_{i=1}^{p_{3}-1} \upgamma_{3i} - \sum_{i=1}^{p_{2}-1} \upgamma_{2i} + B + b  =0  \] and hence the result follows.
	\end{proof}
	
	\begin{remark}\label{gammainDelta}
		If  $\boldsymbol{\upgamma} \notin \Delta,$ then $\Rep(\Lambda_{\boldsymbol{\upgamma}},~\updelta) = \emptyset.$ Certainly, given $\boldsymbol{\upgamma} \notin \Delta,$ then either $\sum_{i=1}^{p_{1}-1} \upgamma_{1i} - \sum_{i=1}^{p_{2}-1} \upgamma_{2i} + A + a \neq 0 ~\text{or}~ \sum_{i=1}^{p_{3}-1} \upgamma_{3i} - \sum_{i=1}^{p_{2}-1} \upgamma_{2i} + B + b \neq 0.$ If $M \in \Rep(\Lambda_{\boldsymbol{\upgamma}},~\updelta),$ then its linear maps between vertices are scalars, which satisfy the relations of $\Lambda_{\boldsymbol{\upgamma}}.$ Now these scalars commute, and as a result of summing up the relations $(1) - (2) + (a) + (c)$ gives $\sum_{i=1}^{p_{1}-1} \upgamma_{1i} - \sum_{i=1}^{p_{2}-1} \upgamma_{2i} + A + a = 0$ and summing the relations $(3) - (2) + (b) + (d)$ gives $\sum_{i=1}^{p_{3}-1} \upgamma_{3i} - \sum_{i=1}^{p_{2}-1} \upgamma_{2i} + B + b = 0,$ which is a contradiction. This is why below we always assume that $\boldsymbol{\upgamma} \in \Delta.$
	\end{remark}
	
	Now consider $\scrM_{\upvartheta_0}(\Lambda_{\boldsymbol{\upgamma}},~\updelta) \coloneqq \Rep(\Lambda_{\boldsymbol{\upgamma}},~\updelta)~ /\!\!\!\!/_{\upvartheta_0}~ \mathrm{GL},$ and define open subsets $V^1_{i,j}, V^2_{i,j}, V^3_{i,j}$ in an identical manner to \S \ref{ModuliofReconAlg}. 
	
	\begin{prop}\label{modulicoverDeformed}
		If  $\boldsymbol{\upgamma} \in \Delta,$ then the following open sets $U^1_{i,j}, U^2_{i,j}, U^3_{i,j},$	cover the moduli space $\scrM_{\upvartheta_0}(\Lambda_{\boldsymbol{\upgamma}},~\updelta).$ 
		
		\noindent
		$\mathbb{C}[U^1_{i,j}]$ is isomorphic to $\mathbb{C}[d_{2i}, u_{2i}, d_{3j}, u_{3j}]$ modulo the ideal generated by the relations
		\begin{align*}
		d_{2i} u_{2i} + \sum_{l=1}^{i-1} \upgamma_{2l} - d_{3j} u_{3j} - \sum_{l=1}^{j-1} \upgamma_{3l} - b,\\
		1 -  d_{2i}(d_{2i}u_{2i} - \upgamma_{2i}) \hdots (d_{2i}u_{2i} - \sum_{l=i}^{p_2-1}\upgamma_{2l}) + d_{3j}(d_{3j}u_{3j} - \upgamma_{3j}) \hdots (d_{3j}u_{3j} - \sum_{l=j}^{p_3-1}\upgamma_{3l}).
		\end{align*}
		
		\noindent
		$\mathbb{C}[U^2_{i,j}]$ is isomorphic to $\mathbb{C}[d_{1i}, u_{1i}, d_{3j}, u_{3j}]$ modulo the ideal generated by the relations 
		
		\begin{align*}  
		d_{1i} u_{1i} + \displaystyle\sum_{l=1}^{i-1} \upgamma_{1l}  - d_{3j} u_{3j} - \displaystyle\sum_{l=1}^{j-1} \upgamma_{3l} - (b-a),\\ d_{1i}(d_{1i}u_{1i} - \upgamma_{1i}) \hdots (d_{1i}u_{1i} - \displaystyle\sum_{l=i}^{p_1-1}\upgamma_{1l}) -  1 + d_{3j}(d_{3j}u_{3j} - \upgamma_{3j}) \hdots (d_{3j}u_{3j} - \displaystyle\sum_{l=j}^{p_3-1}\upgamma_{3l}).
		\end{align*}
		
		\noindent
		$\mathbb{C}[U^3_{i,j}]$ is isomorphic to $\mathbb{C}[d_{1i}, u_{1i}, d_{2j}, u_{2j}]$ modulo the ideal generated by the relations 
		
		\begin{align*}  
		d_{2j} u_{2j} + \displaystyle\sum_{l=1}^{j-1} \upgamma_{2l} - d_{1i} u_{1i} - \displaystyle\sum_{l=1}^{i-1} \upgamma_{1l} - a,\\ 
		d_{1i}(d_{1i}u_{1i} - \upgamma_{1i}) \hdots (d_{1i}u_{1i} - \displaystyle\sum_{l=i}^{p_1-1}\upgamma_{1l}) -  d_{2j}(d_{2j}u_{2j} - \upgamma_{2j}) \hdots (d_{2j}u_{2j} - \displaystyle\sum_{l=j}^{p_2-1}\upgamma_{2l}) + 1.
		\end{align*}

	\end{prop}
	\normalsize
	\begin{proof}The fact these cover the moduli space is identical to Lemma \ref{Cover}. To compute these opens for $\mathbb{C}[U^1_{i,j}],$ consider the relations in  Definition \ref{DeformedReconAlgbera}. Then the relations $(1)$ become 
		\[
		u_{1i} - u_{1i+1} =   \upgamma_{1i},\\
		\]
	
		which implies that 
		\begin{align*}
		u_{12} &=   u_{11} - \upgamma_{11}\\
		u_{13} &=  u_{11}- \upgamma_{11}-  \upgamma_{12}\\
		&\vdots\\
		u_{1p_1} &=  u_{11} -  \displaystyle\sum_{l=1}^{p_1-1}\upgamma_{1l}. 
		\end{align*}
		Therefore, the arrows $u_{1i}$ can be expressed in terms of $u_{11}.$
		\noindent
		The relations $(2)$ become
		\begin{align*}
		u_{21} - u_{22} &=    \upgamma_{21}\\
		u_{22} - u_{23} &=    \upgamma_{22}\\
		&\vdots\\
		u_{2i}d_{2i} - d_{2i+1} &=    \upgamma_{2i}\\
		&\vdots\\
		d_{2p_2-2} - d_{2p_2-1} &=    \upgamma_{2p_2-2}\\
		d_{2p_2-1} - d_{2p_2} &=    \upgamma_{2p_2-1}.
		\end{align*}
		Re-writing the middle as $ d_{2i+1} =  u_{2i}d_{2i} -  \upgamma_{2i},$ working up and down gives 
		\begin{align*}
		u_{21}  &=    u_{2i}d_{2i} + \displaystyle\sum_{l=1}^{i-1}\upgamma_{2l}\\
		&\vdots\\
		u_{2i-1}  &=  u_{2i}d_{2i} +  \upgamma_{2i-1}\\
		d_{2i+1} &=  u_{2i}d_{2i} -  \upgamma_{2i}\\
		d_{2i+2} &=  u_{2i}d_{2i} -  \upgamma_{2i}-  \upgamma_{2i+1}\\
		&\vdots\\
		d_{2p_2}  &= u_{2i}d_{2i} -  \displaystyle\sum_{l=i}^{p_2-1}\upgamma_{2l},
		\end{align*}
		and so all the arrows in arm 2 are determined by  $(u_{2i}, d_{2i}).$ In a similar way, the relations $(3)$ become
		\begin{align*}
		u_{31}  &=    u_{3j}d_{3j} + \displaystyle\sum_{l=1}^{j-1}\upgamma_{3l}\\
		&\vdots\\
		u_{3j-1}  &=  u_{3j}d_{3j} +  \upgamma_{3j-1}\\
		d_{3j+1} &=  u_{3j}d_{3j} -  \upgamma_{3j}\\
		d_{3j+2} &=  u_{3j}d_{3j} -  \upgamma_{3j}-  \upgamma_{3j+1}\\
		&\vdots\\
		d_{3p_3}  &= u_{3j}d_{3j} -  \displaystyle\sum_{l=j}^{p_3-1}\upgamma_{3l},
		\end{align*}
		and all the arrows in arm 3 are determined by  $(u_{3j}, d_{3j}).$
		\noindent
		The remaining relations
		\begin{align*} 
		d_{21}u_{21} - d_{11}u_{11}& =  a\\
		d_{21}u_{21} - d_{31}u_{31}& =  b\\
		u_{1p_1}d_{1p_1} - u_{2p_2}d_{2p_2}& =  A\\
		u_{3p_3}d_{3p_3} - u_{2p_2}d_{2p_2}& = B\\
		d_{11} \hdots d_{1p_1} - d_{21} \hdots d_{2p_2}& + d_{31} \hdots d_{3p_3}= 0,
		\end{align*}
		then become
		\begin{align*} 
		u_{2i}d_{2i} + \displaystyle\sum_{l=1}^{i-1} \upgamma_{2l} - u_{11}& =  a \hspace{3cm}   \hfill{(a)}\\
		(u_{2i}d_{2i} + \displaystyle\sum_{l=1}^{i-1} \upgamma_{2l}) - (u_{3j}d_{3j} + \displaystyle\sum_{l=1}^{j-1} \upgamma_{3l})& =  b\hspace{3.05cm}   \hfill{(b)}\\
		(u_{11} -  \displaystyle\sum_{l=1}^{p_1-1}\upgamma_{1l}) - (u_{2i}d_{2i} - \displaystyle\sum_{l=i}^{p_2-1} \upgamma_{2l})& =  A \hspace{3cm}   \hfill{(c)}\\
		(u_{3j}d_{3j} - \displaystyle\sum_{l=j}^{p_3-1} \upgamma_{3l}) - (u_{2i}d_{2i} - \displaystyle\sum_{l=i}^{p_2-1} \upgamma_{2l})& = B \hspace{3cm}   \hfill{(d)}
		\end{align*} 
		\[1 -  d_{2i}(d_{2i}u_{2i} - \upgamma_{2i}) \hdots (d_{2i}u_{2i} - \displaystyle\sum_{l=i}^{p_2-1}\upgamma_{2l}) + d_{3j}(d_{3j}u_{3j} - \upgamma_{3j}) \hdots (d_{3j}u_{3j} - \displaystyle\sum_{l=j}^{p_3-1}\upgamma_{3l})= 0. \hfill{(x)} \]
		Since $\boldsymbol{\upgamma} \in \Delta,$ $(a) + (c) = 0$ and $(b) + (d) = 0.$ Thus $(a)$ and $(b)$ are satisfied automatically from $(c)$ and $(d).$ By Notation \ref{Notationrel} $(c)-(d) = 0,$ so from $(b)$ and $(x)$ all relations are satisfied.
		
		\smallskip
		\noindent
	 The proof of $\mathbb{C}[U^2_{i,j}]$ and $\mathbb{C}[U^3_{i,j}]$ are similar to that of $\mathbb{C}[U^1_{i,j}]$ above.
	\end{proof}
	
	The following two results will be used in showing that $\scrM_{\upvartheta_0}(\Lambda_{\boldsymbol{\upgamma}},~\updelta)$ is smooth.
	\begin{lemma}\label{Jacobianpr}
		In $\mathbb{C}[x, y],$ for any $f = x (xy - \upalpha_{1}) \hdots (xy - \upalpha_{n})$ with $\upalpha_{i} \in \mathbb{C},$ 
		\[x \frac{\partial f}{\partial x} - y \frac{\partial f}{\partial y} = f(x, y).\]
	\end{lemma}
	
	\begin{proof}
		Set $h_i = (xy - \upalpha_{i})$ for $i=1, 2, \hdots, n,$ and
		\[g \coloneqq (xy - \upalpha_{1}) \hdots (xy - \upalpha_{n}) = h_1 \hdots h_n,\] so that $f = xg.$ Using the product rule, 
		
		\begin{align*}
		\frac{\partial f}{\partial x} =  g + x \cdot \frac{\partial g}{\partial x} & \quad \text{and} \quad \frac{\partial f}{\partial y} =  x \cdot \frac{\partial g}{\partial y}.
		\end{align*}
		Thus the statement follows if $x^2 \cdot \frac{\partial g}{\partial x} -   xy \cdot \frac{\partial g}{\partial y} =0,$ which holds provided that
		$x \cdot \frac{\partial g}{\partial x} -   y \cdot \frac{\partial g}{\partial y} =0.$ Now, again by the product rule
		
		\begin{align*}
		x \cdot\frac{\partial g}{\partial x} &= x \cdot(y h_2 \hdots h_n + h_1 y h_3 \hdots h_n + \hdots + h_1 \hdots h_{n-1} y)\\
		y \cdot\frac{\partial g}{\partial y} &= y \cdot(x h_2 \hdots h_n + h_1 x h_3 \hdots h_n + \hdots + h_1 \hdots h_{n-1} x)
		\end{align*}
		which clearly satisfies $x \cdot \frac{\partial g}{\partial x} -   y \cdot \frac{\partial g}{\partial y} =0.$
	\end{proof}

	\begin{lemma}\label{f2_not_aunit}
		For any scalars $\upalpha_1, \upalpha_2 \hdots, \upalpha_n, \upbeta_2, \hdots, \upbeta_m,$ set $f_1 \coloneqq 
		ab - xy + \upalpha_1$ and $f_2 = 1 - a(ab - \upalpha_2) \hdots (ab - \upalpha_n)  + x(xy - \upbeta_2) \hdots (xy - \upbeta_m).$ Then $\mathbb{C}[a, b, x, y]/(f_1,~f_2) \neq 0.$ 	
	\end{lemma}
	
	\begin{proof}
		Set $\mathcal{A}= \mathbb{C}[a, b, x, y]/(f_1,~f_2),$ then we claim that $\mathcal{A}/(a-1) \neq 0.$ Indeed,
		
		\begin{align*}
		\frac{\mathcal{A}}{(a-1)} &= \frac{\mathbb{C}[b, x, y]}{\left(\begin{array}{c}
			b = xy - \upalpha_1,\\
			1 - (b - \upalpha_2) \hdots (b - \upalpha_n)  + x(xy - \upbeta_2) \hdots (xy - \upbeta_m)		
			\end{array}\right)}\\
		&= \frac{\mathbb{C}[x, y]}{
			(1 - (xy - (\upalpha_1 + \upalpha_2)) \hdots (xy - (\upalpha_1 +\upalpha_n))  + x(xy - \upbeta_2) \hdots (xy - \upbeta_m))		
		},
		\end{align*}
		which is nonzero since the highest degree term is either $x^{n-1}y^{n-1}$ or $x^{m}y^{m-1},$ and these can never cancel. 
	\end{proof}
	
	The following is the main result of this subsection.
	\begin{cor}\label{ModuliDeformedReconSmooth}
		$\scrM_{\upvartheta_0}(\Lambda_{\boldsymbol{\upgamma}},~\updelta)$  is two-dimensional and smooth. 
	\end{cor}
	
	\begin{proof}
		By Proposition \ref{modulicoverDeformed}, $\Rep(\Lambda_{\boldsymbol{\upgamma}},~\updelta)~ /\!\!\!\!/_{\upvartheta_0} \mathrm{GL}$ is covered by the open sets $U^1_{i,j}, U^2_{i,j}, U^3_{i,j},$ so it suffices to prove that each of these is two-dimensional and smooth. \\
		\emph{Case 1.} 
		Note that by  Proposition \ref{modulicoverDeformed}, $\mathbb{C}[U^1_{i,j}] \cong \mathbb{C}[d_{2i}, u_{2i}, d_{3j}, u_{3j}]/(f_1, f_2)$ where
		\begin{align*}
		f_1 & \coloneqq 
		d_{2i} u_{2i} + \displaystyle\sum_{l=1}^{i-1} \upgamma_{2l} - d_{3j} u_{3j} - \displaystyle\sum_{l=1}^{j-1} \upgamma_{3l} - b
		,\\
		f_2 & \coloneqq 1 -  f_{21} + f_{22},\\
		f_{21} & \coloneqq d_{2i}(d_{2i}u_{2i} - \upgamma_{2i}) \hdots (d_{2i}u_{2i} - \displaystyle\sum_{l=i}^{p_2-1}\upgamma_{2l}),\\
		f_{22} & \coloneqq  d_{3j}(d_{3j}u_{3j} - \upgamma_{3j}) \hdots (d_{3j}u_{3j} - \displaystyle\sum_{l=j}^{p_3-1}\upgamma_{3l}).
		\end{align*}
		
		Since $\mathbb{C}[d_{2i}, u_{2i}, d_{3j}, u_{3j}]$ is an affine domain, and $f_1$ is not a unit, then from \cite[Corollary 13.11]{Eisenbud}, $S \coloneqq \mathbb{C}[d_{2i}, u_{2i}, d_{3j}, u_{3j}]/(f_1)$ is a $3$-dimensional domain. Now $f_2$ is not a unit in $S$ since if it were a unit in $S,$ then
		\[ 0 = S/ (f_2) \cong  \mathbb{C}[d_{2i}, u_{2i}, d_{3j}, u_{3j}]/(f_1, f_2) \]
		which would contradict Lemma \ref{f2_not_aunit}. Thus since $S$ is an affine domain, again \cite[Corollary 13.11]{Eisenbud} asserts that 
		\[\text{dim}~ \mathbb{C}[U^1_{i,j}] = \text{dim}~ S - 1 = 2.\]
		Therefore, $\text{dim}~ \scrM_{\upvartheta_0}(\Lambda_{\boldsymbol{\upgamma}},~\updelta) = 2.$\\
		
		\noindent
		Now set $K = (f_1,~ f_2),$ and consider Jacobian matrix
		
		\[\mathcal{J} = \begin{pmatrix}\label{Jacobianminors}
		\frac{\partial f_1}{\partial u_{2i}}&  &\frac{\partial f_1}{\partial d_{2i}}&  &\frac{\partial f_1}{\partial u_{3j}}&  &\frac{\partial f_1}{\partial d_{3j}} \\
		&&&&&&\\
		\frac{\partial f_2}{\partial u_{2i}}&  &\frac{\partial f_2}{\partial d_{2i}}&  &\frac{\partial f_2}{\partial u_{3j}}&  &\frac{\partial f_2}{\partial d_{3j}} 
		\end{pmatrix} = \begin{pmatrix}
		d_{2i}&  &u_{2i}&  & -d_{3j}&  & -u_{3j} \\
		&&&&&&\\
		\frac{\partial f_2}{\partial u_{2i}}&  &\frac{\partial f_2}{\partial d_{2i}}&  &\frac{\partial f_2}{\partial u_{3j}}&  &\frac{\partial f_2}{\partial d_{3j}} 
		\end{pmatrix}.\] 
		Write $\mathrm {J}$ for the ideal generated by the $2 \times 2$ minors of $\mathcal{J}$, together with $K.$ Then
		\begin{align*}
		f_{21} &= d_{2i} \cdot \frac{\partial f_{21}}{\partial d_{2i}} -    u_{2i} \cdot \frac{\partial f_{21}}{\partial u_{2i}} \tag{by Lemma \ref{Jacobianpr}}\\
		&= d_{2i} \cdot \frac{\partial f_2}{\partial d_{2i}} -    u_{2i} \cdot \frac{\partial f_2}{\partial u_{2i}} \tag{by inspection}\\
		&= \frac{\partial f_1}{\partial u_{2i}} \cdot \frac{\partial f_2}{\partial d_{2i}} -    \frac{\partial f_1}{\partial d_{2i}} \cdot \frac{\partial f_2}{\partial u_{2i}},
		\end{align*}  
		and thus $f_{21} \in \mathrm {J}.$\\
		
		\noindent
		With a similar argument as above, using the last  $2 \times 2$ minor
		\[\frac{\partial f_1}{\partial u_{3j}} \cdot \frac{\partial f_2}{\partial d_{3j}} -  \frac{\partial f_1}{\partial d_{3j}} \cdot \frac{\partial f_2}{\partial u_{3j}} = \left( -d_{3j} \right) \cdot \frac{\partial f_{22}}{\partial d_{3j}} -  \left( -u_{3j} \right) \cdot \frac{\partial f_{22}}{\partial u_{3j}}, \]
		\noindent
		together with Lemma \ref{Jacobianpr},
		\[ f_{22} = d_{3j}  \cdot \frac{\partial f_{22}}{\partial d_{3j}} - u_{3j}  \cdot \frac{\partial f_{22}}{\partial u_{3j}}  \in \mathrm {J}.  \]
		\noindent
		Thus, since $f_2 = 1 - f_{21} + f_{22}$ is in $\mathrm {J},$ it follows that $1 \in \mathrm {J}.$ As is standard, see e.g \cite[Algorithm 5.7.6]{Greuel}, this implies that $\scrM_{\upvartheta_0}(\Lambda_{\boldsymbol{\upgamma}},~\updelta)$ is smooth. 
		
		\smallskip
		\noindent
		\emph{Case 2.} The proof of $\mathbb{C}[U^2_{i,j}]$ and $\mathbb{C}[U^3_{i,j}]$ are similar to that of $\mathbb{C}[U^1_{i,j}]$ in above.
	\end{proof}

	\section{Simultaneous Resolution}\label{SimResSection}
	This section considers the invariant representation variety associated to the quiver of the reconstruction algebra in \S\ref{Preliminaries}, and finds its generators in terms of cycles, the deformed reconstruction algebra introduced in \S\ref{Deformed} is used to achieve simultaneous resolution.

	\subsection{Representation Variety}
	Below, we say that arrows $p_1,\hdots,p_n$ are \textit{composable} if $h(p_i) = t(p_{i+1}) ~\text{for all}~ i=1, \hdots , n-1.$ 
	
	\begin{lemma}\label{lemcyclesgen}
		Let $Q$ be the quiver in \textnormal{\eqref{ReconAlgebraQuiver}}. With notation as in \S \textnormal{\ref{Generalities}}, $\mathcal{\scrR}^G$ is generated by $p + I$ where $p$ is a cycle in $Q$.
	\end{lemma}
	\begin{proof} Choose a monomial $p = p_1 \hdots p_n \in \scrR,$ where ${p_i}$'s are arrows. We claim that
		$\upmu \cdot (p + I)  = p~+I ~ \text{for all} ~ \upmu \Leftrightarrow p$ is a cycle. First observe that $\upmu \cdot (p + I) = \upmu \cdot p + I
		= (\upmu_{t(p_1)} \hdots \upmu_{t(p_n)})^{-1} p (\upmu_{h(p_1)} \hdots \upmu_{h(p_n)}) + I$.
		
		\noindent		
		($\Leftarrow$) If $p$ is a cycle, in particular it is composable. Thus for all $\upmu \in G,$ 
		\begin{align*}
		\upmu \cdot (p + I) &= \upmu \cdot p + I\\
		&= \upmu_{t(p_1)}^{-1} p_1\upmu_{h(p_1)}\upmu_{t(p_2)}^{-1}p_2\upmu_{h(p_2)} \hdots  \upmu_{t(p_n)}^{-1}p_n \upmu_{h(p_n)}+ I\\
		&= \upmu_{t(p_1)}^{-1} \upmu_{h(p_n)} p_1 p_2  \hdots p_n + I\\
		&= \upmu_{t(p_1)}^{-1} \upmu_{h(p_n)} p + I\\
		&= p + I.  \tag{since $t(p_1) = h(p_n)$}
		\end{align*}
		Hence $ p + I\in \scrR^G$.
		
		\noindent		
		($\Rightarrow$)
		Suppose that $p + I \in \scrR^G$ such that $\upmu \cdot (p + I)  = \upmu \cdot p + I = p~+I$ for all $\upmu$. Then $\upmu_{h(p_1)}$ must cancel some $\upmu_{t(p_i)}^{-1}$ for some $i$, so $h(p_1) = t(p_i).$ Now consider $\upmu_{h(p_i)}.$ It must cancel $\upmu_{t(p_j)}^{-1}$ for some $j$, so $h(p_i) = t(p_j).$
		Continuing like this, we can assume $p = p_1p_ip_j \hdots p_m$ where $p_1p_ip_j \hdots p_m$ is composable.
		But then $\upmu \cdot (p + I) = \upmu \cdot p + I = \upmu_{t(p_1)}^{-1} \cdot p \cdot \upmu_{h(p_m)} + I$ and so since $\upmu \cdot (p + I)= p + I$, $t(p_1) = h(p_m),$ and $p$ is a cycle.
	\end{proof}
	
	\subsection{Reconstruction Algebras}\label{ReconAlg} By Lemma \ref{lemcyclesgen}, $\scrR^G$ is generated by cycles. This subsection finds a finite generating set, by considering 
	\begin{equation*} 
	\scrS_1 \coloneqq \{ 2\textnormal{-cycles}\}\bigcup 
	\left\{
	\begin{array}{cccc}
	D_1U_1, &D_1U_2,  &D_1U_3\\
	D_2U_1, &D_2U_2, &D_2U_3\\
	D_3U_1, &D_3U_2,  &D_3U_3
	\end{array}
	\right\}.
	\end{equation*}
	
	\begin{prop}\label{invariantringgens}
		$\scrR^G$ is generated as a $\mathbb{C}$--algebra by the set 	$\scrS_1.$
	\end{prop}
	
	\begin{proof}
		For any vertex $v$, consider a non trivial cycle $p$, then it must leave the vertex. According to the quiver, there are two options:
		
		\smallskip
		\noindent
		\emph{Case 1.} The path $p$ starts with a downward arrow ($p = d_{ij}p^{\prime})$ from vertex $v$ along the $i^{th}$ arm. Since $p$ is a cycle, $p^{\prime} \colon h(d_{ij}) \rightarrow v.$  If $p^{\prime}$ continues downwards, at some stage $p^{\prime}$ stops travelling downwards, and we can write $p = d_{ij}d_{ij+1}\hdots d_{ik} p^\prime$
		for some $p^{\prime}\colon a \rightarrow v.$ If $a \neq b,$ where $b$ is the bottom vertex, then $p^\prime$ starts upwards, so 
		\[p = d_{ij} \hdots \underbrace{(d_{ik}u_{ik})}_{z}p^{\prime\prime} \sim z \text{(cycles of length smaller than p)}\] thus by induction, $p \in \left<\scrS_1\right>$. Hence we can assume $a = b,$ and $p^{\prime}$ must travel up one of the arms. According to the quiver $Q$ in \S\ref{Preliminaries}, there are $3$ options.
		
		\begin{enumerate}
			\item [(a)] $p^{\prime}$ travels up the $i^{th}$ arm. Again $p = d_{ij} \hdots (d_{ip_i}u_{ip_i})p^{\prime\prime},$ by induction we are done. 
			
			\item [(b)] $p^{\prime}$ travels up one of the other $2$ arms. If it doubles back on itself before reaching the top ($0^{th}$ vertex), we are done by induction since $\scrS_1$ contains all $2$-cycles. Hence we can assume $p^\prime$ reaches the top, so $p= d_{ij} \hdots d_{ip_i}U_lp^{\prime\prime}$ for some $l \neq i.$ Repeating, either $p^{\prime\prime}$ travels down the $i^{th}$ arm without doubling back, in which case $p \sim D_iU_l p^{\prime\prime\prime},$ or $p^{\prime\prime}$ travels down the $k^{th}$ arm $(k \neq i)$ without doubling back in which case $p \sim(d_{ij} \hdots d_{ip_i}) (U_lD_k) p^{\prime\prime}.$ In either case, by induction we are done.
		\end{enumerate}

		\smallskip
		\noindent
		\emph{Case 2.}
		The path $p$ starts with an upward arrow. This is very similar to Case 1,
		after interchanging the downward and the upward arrows. 
	\end{proof}

	Now, set
	\begin{align*}
	\scrS_2 & \coloneqq
	\{ 2\textnormal{-cycles}\}\bigcup 
	\left\{
	\begin{array}{ccccc}
	&D_1U_2, &D_1U_3\\
	D_2U_1, &&D_2U_3\\
	\end{array}
	\right\},
	\\
	\scrS_3 & \coloneqq \{ 2\textnormal{-cycles}\}\bigcup\,\, \{D_1U_2,~D_2U_1,~D_2U_3\}.
	\end{align*}

	\begin{lemma}\label{lemmainvargens}
		$\left<\scrS_3\right> = \left<\scrS_2\right> = \left<\scrS_1\right>.$
	\end{lemma}
	
	\begin{proof}
		Since $\scrS_3 \subseteq \scrS_2 \subseteq \scrS_1,$ it suffices to prove that $\left<\scrS_1\right> \subseteq \left<\scrS_2\right> \subseteq \left<\scrS_3\right>.$
		For $\left<\scrS_1\right> \subseteq \left<\scrS_2\right>,$ multiplying the relation $D_3 =  D_2 - D_1$ with $U_1, U_2 , U_3,$ shows that all elements in $\scrS_1$ can be generated by the elements in $\scrS_2.$ 
		
		\bigskip
		\noindent
		For $\left<\scrS_2\right> \subseteq \left<\scrS_3\right>,$ again $D_3U_i =  D_2 U_i - D_1 U_i$ for $i = 1,~2,~3$ and so $D_1 U_i =  D_2 U_i - D_3U_i$ where  $D_iU_i$ is a product of $2$-cycles. Therefore,  $D_3 U_1,~D_3 U_2,~D_1 U_3$ can be expressed in terms of elements of $\scrS_3.$ 
		
	\end{proof}

	\subsection{Simultaneous Resolution}\label{SimRes} Recall from 
	(\ref{eqnRep}) that $\mathcal{\scrR} = \mathbb{C}[\Rep(\mathbb{C}Q,\updelta)]/I,$ and $G = \mathrm{GL}$ acts on $\mathcal{\scrR},$ so we can form $\mathcal{\scrR}^G.$  By \S\ref{ReconAlg}, $\mathcal{\scrR}^G$ is generated by
	\begin{align*}
	\mathsf{w}_1 &\coloneqq D_1U_2 + I\\
	\mathsf{w}_2 &\coloneqq ~~D_2U_1 + I\\
	\mathsf{w}_3 &\coloneqq -D_2U_3 + I\\				
	\mathsf{v}_{i,j} &\coloneqq ~~d_{ij}u_{ij} + I.
	\end{align*}
	for $i=1, 2, 3 ~\text{and}~ 1 <j \leq p_i.$	\\
	
	\noindent 
	Write $(\upbeta_{1}, \upbeta_{2}, \upbeta_{3}, \upalpha_{1,1},\hdots, \upalpha_{1,p_1}, \upalpha_{2,1},\hdots, \upalpha_{2,p_2}, \upalpha_{3,1},\hdots,$ $\upalpha_{3,p_3})$ for the point in Spec $\mathcal{\scrR}^G$ corresponding to the maximal ideal $(\mathsf{w}_1 -\upbeta_{1}, \mathsf{w}_2 -\upbeta_{2}, \mathsf{w}_3 -\upbeta_{3},  \mathsf{v}_{1,1}-\upalpha_{1,1},\hdots,\mathsf{v}_{3,p_3}-\upalpha_{3,p_3}).$ Let $Q$ be the quiver of the reconstruction algebra, and consider the map
	\[ \uppi \colon \text{Spec}~ \mathcal{\scrR}^G =  \Rep(\mathbb{C}Q/I,~\updelta)~ /\!\!\!\!/~ \mathrm{GL}  \rightarrow \Delta,\]
	defined by taking 
	\[
	\begin{tikzpicture}
	\node (A) at (0,0) {$(\upbeta_{1}, \upbeta_{2}, \upbeta_{3}, \upalpha_{1,1},\hdots, \upalpha_{1,p_1}, \upalpha_{2,1},\hdots, \upalpha_{2,p_2}, \upalpha_{3,1},\hdots, \upalpha_{3,p_3})$};
	\scriptsize
	\node (a) at (0,-2) {$((\upalpha_{1,i} - \upalpha_{1,i+1})^{p_1-1}_{i=1},(\upalpha_{2,i} - \upalpha_{2,i+1})^{p_2-1}_{i=1},(\upalpha_{3,i} - \upalpha_{3,i+1})^{p_3-1}_{i=1}, \upalpha_{2,1} - \upalpha_{1,1}, \upalpha_{2,1} - \upalpha_{3,1}, \upalpha_{1,p_1} - \upalpha_{2,p_2},  \upalpha_{3,p_3} - \upalpha_{2,p_2}) .$};
	\draw[|->] (A)--(a);
	\end{tikzpicture}
	\]
	
	\begin{remark}\label{fibreabovegamma}
		The fibre above a point $\boldsymbol{\upgamma} \in \Delta$ is  precisely $\Rep(\Lambda_{\boldsymbol{\upgamma}},~\updelta) /\!\!/ \mathrm{GL}.$ Indeed, the fibre above $\boldsymbol{\upgamma} \in \Delta$ is the zero locus of 
		\[
		\begin{array}{ccccc}
		(1)&~\mathsf{v}_{1,i} - \mathsf{v}_{1,i+1} &=  &\upgamma_{1,i} &~ \text{for all}~ 1 \leq i \leq p_1-1\\
		(2)&~\mathsf{v}_{2,i} - \mathsf{v}_{2,i+1} &=  &\upgamma_{2,i} &~ \text{for all}~ 1 \leq i \leq p_2-1\\
		(3)&~\mathsf{v}_{3,i} - \mathsf{v}_{3,i+1} &=  &\upgamma_{3,i}&~ \text{for all}~ 1 \leq i \leq p_3-1\\
		(a)&~\mathsf{v}_{2,1} - \mathsf{v}_{1,1} &=  &a&\\
		(b)&~\mathsf{v}_{2,1} - \mathsf{v}_{3,1} &=  &b&\\
		(c)&~\mathsf{v}_{1,p_1} - \mathsf{v}_{2,p_2} &=  &A&\\
		(d)&~\mathsf{v}_{3,p_3} - \mathsf{v}_{2,p_2} &= &B&
		\end{array}\]  
		By Definition \ref{DeformedReconAlgbera}, this equals $\Rep(\Lambda_{\boldsymbol{\upgamma}},~\updelta) /\!\!/ \mathrm{GL}.$ In particular, the fibre above the origin is $\Rep(\Lambda_{\boldsymbol{0}},~\updelta) /\!\!/ \mathrm{GL}.$ Since $\Lambda_{\boldsymbol{0}}$ is the (undeformed) reconstruction algebra, this is known to be the determinantal singularity corresponding to (\ref{s Veron dual graph}).
	\end{remark}
	
	\noindent
	The following is the main result of this paper.
	
	\begin{theorem}\label{thm: main}
		The diagram 
		\[
		\begin{tikzpicture}
		\node (A) at (0,0) {$\Rep(\mathbb{C}Q/I,~\updelta)~ /\!\!\!\!/_{\upvartheta_0} \mathrm{GL}$};
		\node (B) at (4,0) {$\Rep(\mathbb{C}Q/I,~\updelta)~ /\!\!\!\!/~ \mathrm{GL}$};
		\node (b) at (4,-2) {$\Delta$};
		\draw[->] (A)-- node[above]  {} (B);
		\draw[densely dotted,->] (A)-- node[below]  {$\upphi$} (b);
		\draw[->] (B)-- node[right]  {$\uppi$} (b);
		a\end{tikzpicture}
		\]
		is a simultaneous resolution of singularities in the sense that the morphism $\upphi$ is smooth, and $\uppi$ is flat.
	\end{theorem}
	
	\begin{proof}
		Write $\upphi$ for the composition
		\[Y= \Rep(\mathbb{C}Q/I,~\updelta)~ /\!\!\!\!/_{\upvartheta_0} \mathrm{GL} \rightarrow \Rep(\mathbb{C}Q/I,~\updelta)~ /\!\!\!\!/~ \mathrm{GL} \rightarrow \Delta.\]
		We first claim that $\upphi$ is flat. Since $(1)~\Delta$ is regular, $(2)~Y$ is regular (so Cohen-Macaulay) by Corollary \ref{ModuliSmooth}, $(3)~\mathbb{C}$ is algebraically closed so $\upphi$ takes closed points of $Y$ to closed points of $\Delta,$ and $(4)$ for every closed point $\boldsymbol{\upgamma} \in \Delta,$ for the same reason as in Remark \ref{fibreabovegamma} the fibre $\upphi^{-1}(\boldsymbol{\upgamma})$ is $\Rep(\Lambda_{\boldsymbol{\upgamma}},~\updelta)~ /\!\!\!\!/_{\upvartheta_0} \mathrm{GL}$  which is always two-dimensional by Corollary \ref{ModuliDeformedReconSmooth}, it follows  from \cite[Corollary to 23.1]{Matsumura} that $\upphi$ is flat.
		
		Now as in {\cite[3.35]{Liu}} to show that $\upphi$ is smooth, we just require smoothness (equivalently regularity, as we are working over $\mathbb{C}$) at closed points of fibres above closed points $\boldsymbol{\upgamma} \in \Delta.$ But as above $\upphi^{-1}(\boldsymbol{\upgamma})$ is $\Rep(\Lambda_{\boldsymbol{\upgamma}},~\updelta)~ /\!\!\!\!/_{\upvartheta_0} \mathrm{GL},$ which is regular at all closed points by Corollary \ref{ModuliDeformedReconSmooth}. Thus $\upphi$ is a smooth morphism, as required. 
		
		Finally, the above can be adapted to show that $\uppi$ is flat. We have that $\uppi^{-1}(\boldsymbol{\upgamma})=\Rep(\Lambda_{\boldsymbol{\upgamma}},~\updelta) /\!\!/ \mathrm{GL},$ which is always two-dimensional as a consequence of the resolution of its singularities computed in Corollary \ref{ModuliDeformedReconSmooth}. Thus we can still appeal to \cite[Corollary to 23.1]{Matsumura}.
	\end{proof}

	\section{The Representation Variety in Determinantal Form}\label{repvar}
	This section gives a conjectural explicit presentation of the high-dimensional invariant representation variety $\mathcal{\scrR}^G = \Rep(\mathbb{C}Q/I,~\updelta)/\!\!/ \mathrm{GL}$ in \S\ref{SecSimRes}. Although not strictly needed for simultaneous resolution in \S \ref{SimResSection}, it does link to other work in the literature.
	
	\subsection{Determinantal form}\label{QDetform}
	Consider a $2 \times n$ matrix
	
	\[
	\begin{pmatrix}
	a_1&  &a_2&  &\hdots& & a_n \\
	b_1& &b_2& &\hdots& & b_n \\ 
	\end{pmatrix}\]
	
	Following Stevens \cite[\S12]{DeformationsofSing}, consider the $2 \times 2$ \textit{minors} of this $2 \times n$ \textit{matrix}, which for all, $i < j$ are defined to be
	\[ a_i \cdot b_j - b_i \cdot a_j. \]

	Consider the natural homomorphism
	\[
	\mathbb{C}[\mathsf{w} , \mathsf{v}] \xrightarrow{\varphi} \scrR^G = \Rep(\mathbb{C}Q/I,~\updelta)/\!\!/ \mathrm{GL}, 
	\]
	defined by taking
	
	\begin{align*}
	\mathsf{w}_1 &\mapsto D_1U_2 + I\\
	\mathsf{w}_2 &\mapsto ~~D_2U_1 + I\\
	\mathsf{w}_3 &\mapsto -D_2U_3 + I\\				
	\mathsf{v}_{i,j} &\mapsto ~~d_{ij}u_{ij} + I.
	\end{align*}
for $i=1, 2, 3 ~\text{and}~ 1 <j \leq p_i.$	
	
	\begin{prop}
		The homomorphism $\mathbb{C}[\mathsf{w} , \mathsf{v}] \xrightarrow{\varphi} \scrR^G$ is surjective, and all the $2 \times 2$ minors of the matrix 
		\begin{equation}\label{R matrix}
		\left(
		\begin{array}{ccccc}
		{\mathsf w}_2&{\mathsf w}_3 &{v_{2,1}\hdots v_{2,{p_2}}}\\
		{v_{1,1}\hdots v_{1,{p_1}}}&{\mathsf w}_3+{v_{3,1}\hdots v_{3,{p_3}}}&{\mathsf w}_1
		\end{array}
		\right)
		\end{equation}
		are all sent to zero.
	\end{prop}
	
	\begin{proof}
		Surjectivity follows from Lemma \ref{lemmainvargens}. The $2 \times 2$ minor of the outer columns gives
		\[
		\varphi (\mathsf{w}_2 \mathsf{w}_1 - (\mathsf{v}_{1,1}\hdots \mathsf{v}_{1,{p_1}}) (\mathsf{v}_{2,1}\hdots \mathsf{v}_{2,{p_2}}))=
		D_2U_1D_1U_2 - D_1U_1D_2U_2  =0\]
		\noindent
		Using $D_1U_i =  D_2 U_i - D_3 U_i$ from the relation $x_1^{p_1} =  x_2^{p_2} - x_3^{p_3},$ we further have
		\begin{align*}
		\varphi ( \mathsf{w}_1\mathsf{w}_3) &= (D_1U_2)(-D_2U_3)\\
		&= -(D_1U_3)(D_2U_2)\\
		&= (-D_2U_3 +D_3U_3)D_2U_2\\
		&= \varphi ((\mathsf{w}_3 + \mathsf{v}_{3,1}\hdots \mathsf{v}_{3,p_3})(\mathsf{v}_{2,1}\hdots \mathsf{v}_{2,p_2})).
		\end{align*}
		Thus the $2 \times 2$ minor consisting of the $i^{th}$ column and the last one is sent to zero.\\
		
		Finally,
		\begin{align*}
		\varphi (\mathsf{w}_2(\mathsf{w}_3 + \mathsf{v}_{3,1}\hdots \mathsf{v}_{3,p_3})) &= (D_2U_1)( \mathsf{w}_3 + \mathsf{v}_{3,1}\hdots \mathsf{v}_{3,p_3})\\
		&= (D_2U_1)(- D_2U_3 +D_3U_3)\\
		&= (D_2U_1) \cdot -(D_1U_3)\\
		&= (D_1U_1)(-D_2U_3)\\
		&= \varphi ((\mathsf{v}_{1,1}\hdots \mathsf{v}_{1,p_1})\mathsf{w}_3).\\
		\end{align*}

		This shows that the $2 \times 2$ minors of the matrix belong to the kernel of $\varphi,$ as required. 
	\end{proof}
	
	\begin{conj}
		The kernel of $\varphi$ is given by the ideal generated by the $2\times 2$ minors of the matrix \eqref{R matrix}.
	\end{conj}
	Calculations in magma \cite{magma} confirm this in small cases, but the Gr\"obner approach seems hard in general.

\end{document}